\documentclass[12pt]{amsart}
\usepackage{amsfonts}
\usepackage{graphicx,amssymb,amsfonts,amsmath,amscd}
\usepackage[all,ps,cmtip]{xy}
\usepackage{amscd}

\usepackage{mathrsfs}
\let\mathcal\mathscr

\usepackage{hyperref}

\def\dra{\dashrightarrow}

\def\Alb{\mathop{\rm Alb}\nolimits}

\def\dim{\mathop{\rm dim}\nolimits}

\def\Ker{\mathop{\rm Ker}\nolimits}
\def\rank{\mathop{\rm rank}\nolimits}

\def\Pic{\mathop{\rm Pic}\nolimits}

\def\tilde{\widetilde}

\def\llra{\hbox to 12mm{\rightarrowfill}}

\def\moins{\mathop{\hbox{\vrule height 3pt depth -2pt
width 5pt}\,}}

\def\cO{{\mathcal O}}

\def\cQ{{\mathcal Q}}

\def\1Y{{Y^{'}}}

\def\KK{{\widetilde{K}}}

\newtheorem{theo}{Theorem}[section]

\newtheorem{prop}[theo]{Proposition}
\newtheorem{lemm}[theo]{Lemma}
\newtheorem{coro}[theo]{Corollary}

\newtheorem{defi}[theo]{Definition}
\newtheorem{rema}[theo]{Remark}

\newtheorem{exam}[theo]{Example}

\title{Varieties with $q(X)=\dim(X)$ and $P_2(X)=2$}
\date{\today}
\author{Zhi Jiang}
\email{jiang@dma.ens.fr}

\begin{document}
\maketitle

\begin{abstract}
 We give a complete description  of all smooth projective complex   varieties with $q(X)=\dim(X)$ and $P_2(X)=2$.
\end{abstract}
 
It has always been a goal of algebraic geometers to classify algebraic varieties and a good part of their work in the   twentieth century was devoted to the classification of algebraic surfaces. Of course, one can obtain a complete description only  in particular cases, such as when the numerical invariants of the surface are small.

In the past thirty years, the development of new techniques made it possible to obtain results in higher dimensions as well. The numerical (birational) invariants of a smooth projective complex   variety $X$ that are commonly used are its {\em irregularity} $q(X):=h^1(X,\cO_X)$ and  its {\em plurigenera} $P_m(X):=h^0(X,\omega_X^m)$.

One can quote for example a beautiful result of Kawamata (\cite{KA}), who proved that  $X$ is birational to an abelian variety if and only if  $q(X)=\dim(X)$ and the Kodaira dimension   $\kappa(X)$ is 0 (this means  $\max_{m>0}P_m(X)=1$). This result has been improved on by many authors, and the optimal version can be found in \cite{CH1} and \cite{jiang}:  {\em $X$ is birational to an abelian variety if and only if  $q(X) = \dim(X)$, and $ P_2(X)=1$ or
$0<P_m(X)\leq m-2$ for some $m\geq 3$.}

When $q(X) = \dim(X)$, but the  numerical invariants of $X$ are a little bit higher than these bounds, one can still obtain a complete birational description of $X$. Hacon and Pardini treated the case $P_3(X) = 2$ and proved that  $X$ is birational to a smooth double cover of its Albanese variety $\Alb(X)$,
 with   explicit and very specific branch locus (\cite{HAC1}). Hacon then gave an equally precise description in the case $P_3(X) = 3$ ($X$ is birational to a smooth bidouble cover of   $\Alb(X)$) in \cite{Ha}, and Chen and Hacon dealt with the case $P_3(X) = 4$, where the description obtained is still complete but more complicated (\cite{CH4}).

In this article, following the strategy  of Chen and Hacon in \cite{CH4}, but building on the results of \cite{jiang}, we give   a complete birational description of $X$ in the case  $P_2(X)=2$ (Theorem \ref{5}).

\medskip\noindent{\bf Theorem } {\em
 Let $X$ be a smooth projective variety with
  $q(X)=\dim (X)$ and $P_2(X)=2$. Then $\kappa(X)=1$ and $X$ is birational to a
quotient $(K\times C)/G$, where $K$ is an abelian variety, $C$ is
a curve, $G$ is a finite group which acts diagonally and freely on
$K\times C$, and $C\rightarrow C/G$ is branched at 2 points.}
 \medskip

A main ingredient in the above classifications is to bound the possible Kodaira dimensions of $X$. In this direction, we have the following result (Theorem \ref{4.11}).

\medskip\noindent{\bf Theorem } {\em
Let $X$ be a smooth projective variety with $q(X)=\dim (X)$ and $0< P_m(X)\leq 2m-2$, for some $m\geq 4$. Then $\kappa(X)\leq 1$.}
 \medskip

Throughout this article, we work over the field of complex numbers.

\section{Preliminaries}

In this section we recall some definitions and prove preliminary results. Let $X$ be a smooth projective variety.

\subsection{Albanese variety}  There is an abelian variety $\Alb(X)$, called the {\em Albanese variety} of $X$, together with a morphism $a_X:X\to \Alb(X)$ called the {\em Albanese morphism} of $X$, which has a universal property for morphisms from $X$ to abelian varieties (\cite{Laz}, \S4.4). We say that $X$ has {\em maximal Albanese dimension} if $\dim(a_X(X))=\dim(X)$.
 We   recall a criterion for the surjectivity of the Albanese morphism (\cite{jiang}, Theorem 2.9).

\begin{theo}\label{J1}Let $X$ be a smooth projective variety. If $$0<P_m(X)\leq 2m-2,$$ for some $m\geq 2$, the
Albanese morphism $a_X: X\rightarrow \Alb(X)$ is surjective.
\end{theo}

\subsection{Cohomological   loci}Let $X$ be a smooth projective variety and let $ F $ be a
coherent sheaf on $X$. The cohomological support loci of $F$ are
defined as $$V_i( F)=\{P\in \Pic^0(X)\mid H^i(X, F\otimes
P)\neq 0\}.$$

\subsection{Iitaka fibration}  We denote by $X\dra I(X)$ the {\em Iitaka fibration} of $X$, where $I(X)$ has dimension $\kappa(X)$  (\cite{Laz}, Definition 2.1.36). Given a surjective morphism $f: X\rightarrow Y$, we say  that {\em the Iitaka model of $X$ dominates $Y$}   if there exist  an integer $N>0$ and an ample divisor $H$ on $Y$ such that $NK_X   \succeq  f^*H$.

The following proposition  is \cite{jiang}, Lemma 2.2.

\begin{prop}\label{J2}Let
$f:X\rightarrow Y$ be a surjective morphism between smooth
projective varieties and assume that the Iitaka model of $X$
dominates $Y$. Fix a torsion element $Q\in \Pic^0(X)$ and an integer $m\geq 2$. Then
$h^0(X, \omega_X^{m}\otimes Q\otimes f^*P)$ is constant for all
$P\in \Pic^0(Y)$.
\end{prop}

We deduce a corollary in the case where $Y$ is a curve, which we will use several times.

\begin{coro}\label{J3}Let $f: X\rightarrow C$ be a surjective morphism
between a smooth projective variety $X$ and a smooth projective curve $C$ of genus $\geq 1$
 and assume  that the Iitaka model of $X$ dominates $C$. If for some torsion
element $Q\in \Pic^0(X)$ and some integer $m\geq 2$, we have $h^0(X,
\omega_X^m\otimes Q)\neq 0$, then $f_*(\omega_X^m\otimes Q)$ is an
ample vector bundle on $C$.
\end{coro}

\begin{proof}Since $C$ is a smooth curve, the torsion-free sheaf $f_*(\omega_X^m\otimes Q)$ is locally free.
Since $Q$ is   torsion, there exists an \'{e}tale cover $\pi: X^{'}\rightarrow X$ such that $\omega_X^m\otimes Q$ is a direct summand of $\pi_*\omega_{X^{'}}^m$. By  \cite{V}, Corollary 3.6, the vector bundle $(f\circ\pi)_*\omega_{X^{'}/C}^m$ is nef, hence so is $f_*(\omega_{X/C}^m\otimes
Q)$.  

If $g(C)\geq 2$, $\omega_C$ is ample, hence so is $f_*(\omega_X^m\otimes Q)$.

If $g(C)=1$, we   claim the following standard fact for which we could not find a reference:
\begin{itemize}
\item[$\clubsuit$]for any nef vector bundle $F$ on an elliptic curve $C$, the cohomological   locus
$V_1(F)$  is finite.
\end{itemize}
We prove the claim by induction on the rank of $F$. The rank-$1$
case is trivial. Let $r>0$ be an integer. Assuming   $\clubsuit$
  proved for all nef vector bundles of rank $\leq r$, we
will prove $\clubsuit$ for any nef vector bundle $F$ of rank $r+1$.

 We consider the
Harder-Narasimhan filtration   (\cite{Laz}, Proposition 6.4.7)
\begin{eqnarray*}
0=F_n\subset F_{n-1}\subset\cdots\subset F_1\subset F_0=F,
\end{eqnarray*}
where $F_i$ are subbundles of $F$ with the properties that
$F_i/F_{i+1}$ is a semistable bundle for each $i$ and
\begin{eqnarray*}
\mu(F_{n-1}/F_n)>\cdots >\mu(F_1/F_2)>\mu(F_0/F_1).
\end{eqnarray*}
Since $F=F_0$ is nef, so is $F_0/F_1$, hence $\mu(F_0/F_1)\geq 0$.
So $F_i/F_{i+1}$ is a semistable vector bundle with positive slope,
for each $i\geq 1$. Hence, for each $i\geq 1$, $F_i/F_{i+1}$ is an
ample vector bundle (see Main Claim in the proof of \cite{Laz}, Theorem 6.4.15). Thus $F_1$ is also ample, and $V_1(F_1)$ is empty.
We just need to prove that $V_1(F_0/F_1)$ is finite.

If $\mu(F_0/F_1)>0$, we again have that $F_0/F_1$ is ample and
$V_1(F)$ is empty; so we are done. If $\mu(F_0/F_1)=0$, take $P\in
V_1(F_0/F_1)$. Then $h^1(C, F_0/F_1\otimes P)\neq 0$. Hence, by
Serre duality,
 there exists a non-trivial homomorphism of bundles $$\pi_P: F_0/F_1\rightarrow P^{\vee}.$$
 Since $F_0/F_1$ is semistable and $\mu(F_0/F_1)=0$, $\pi_P$ is surjective. We have an exact sequence of
 vector bundles: $$0\rightarrow G\rightarrow F_0/F_1\rightarrow P^{\vee}\rightarrow 0.$$
 The rank of $G$ is $\leq r$. Since $F_0/F_1$ is semistable and $\mu(G)=\mu(F_0/F_1)=0$, $G$ is also semistable.
Hence $G$ is a nef vector bundle (by the Main Claim quoted above) of rank $\leq r$ and, by induction, $V_1(G)$ is  finite. We conclude that $V_1(F)=V_1(G)\cup \{P\}$ is finite. We have finished the proof of the Claim.
\medskip

Let the line bundle $Q$ and the integer $m$ be as in the assumptions.
By Proposition \ref{J2}, for $m\geq 2$, $h^0(X, \omega_X^m\otimes Q\otimes
f^*P)=h^0(C, f_*(\omega_X^m\otimes Q)\otimes P)$ is constant for all
$P\in\Pic^0(C)$, hence $h^1(C, f_*(\omega_X^m\otimes Q)\otimes P)$
is also constant for all $P\in\Pic^0(C)$. By the claim $\clubsuit$,
$h^1(C, f_*(\omega_X^m\otimes Q)\otimes P)=0$ for all $P\in
\Pic^0(C)$. Hence $f_*(\omega_X^m\otimes Q)$ is an
I.T. vector bundle of index 0, hence is ample (\cite{D}, Corollary 3.2).
\end{proof}

\subsection{Iitaka fibration of a variety  of maximal Albanese dimension}\label{s14}

We assume in this section that $X$ has maximal Albanese dimension and we consider
a model $f: X\rightarrow Y$  of the Iitaka fibration of $X$, where
 $ Y$ is a smooth projective
variety. We
have the commutative diagram:
\begin{eqnarray}\label{d14}
\xymatrix{ X\ar[d]^f\ar[r]^(.4){a_X}&\Alb(X)\ar[d]^{f_*}\\
Y\ar[r]^(.4){a_Y}&\Alb(Y).}\end{eqnarray} By Proposition 2.1 of
\cite{HAC1}, $a_Y$ is generically finite, $f_*$ is an algebraic
fiber space, $\Ker ( f_*)$ is an abelian variety denoted by $K$, and a
general fiber of $f$ is birational to an abelian variety
$\widetilde{K}$ isogenous to $K$. Let $G$   be the kernel of the group morphism
$$\Pic^0(X)=\Pic^0(\Alb(X))\rightarrow\Pic^0(K)\rightarrow
\Pic^0(\tilde{K}).$$ Then $f^*\Pic^0(Y)$ is contained in $G$, and
$\overline{G}=G/f^*\Pic^0(Y)$  is a finite group
consisting of elements $\chi_1, \ldots, \chi_r$. Let $P_{\chi_1},
\ldots, P_{\chi_r}\in G$ be torsion line bundles representing lifts
of the elements of $\overline{G}$, so that
$$G=\bigsqcup_{i=1}^r(P_{\chi_i}+f^*\Pic^0(Y)).$$
There is an easy observation:

\begin{lemm}\label{5.11}
Under the above assumptions and notation, let moreover $P\in \Pic^0(X)$. If $H^0(X, \omega_X^m\otimes P)\neq 0$ for
some $m>0$, we have $P\in G$.
\end{lemm}

\begin{proof}If $P\notin G$, since a general fiber $F$ of $f$ is
birational to the abelian variety $\widetilde{K}$ and
$P|_{\widetilde{K}}$ is non-trivial, any section of $\omega_X^m\otimes P$ vanishes on $F$. Hence $H^0(X, \omega_X^m\otimes
P)=0$, which is a contradiction.
\end{proof}

Chen and Hacon made
several useful observations about the cohomological locus $V_0(\omega_X)$ (\cite{CH2}, Lemma
2.2)  which we summarize in the following proposition.

\begin{prop}[Chen-Hacon]\label{ch1}Under the above assumptions and notation, we have the following.
\begin{itemize}
\item[\rm (1)] $V_0(\omega_X)\subset G.$
\item[\rm (2)] Denote by $V_0(\omega_X, \chi_i)$ the union of
irreducible components of $V_0(\omega_X)$ contained in
$P_{\chi_i}+f^*\Pic^0(Y)$. Then for each $i$, $V_0(\omega_X,
\chi_i)$ is not empty.
\item[\rm (3)] If $P_{\chi_i}\notin f^*\Pic^0(Y)$, the dimension of $V_0(\omega_X, \chi_i)$ is positive.
\end{itemize}
\end{prop}

Since every component of $V_0(\omega_X)$ is a
translate by a torsion point of an abelian subvariety of $\Pic^0(X)$ (\cite{GL1}, \cite{GL2}, \cite{S}), we can write by item (1):
$$V_0(\omega_X)=\bigcup_{1\leq i\leq r}\bigcup_s(P_{\chi_{i,s}}+T_{\chi_{i,s}})\subset
G,$$ where $P_{\chi_{i,s}}\in P_{\chi_i}+f^*\Pic^0(Y)$ is a torsion
point and $T_{\chi_{i,s}}$ is an abelian subvariety of
$f^*\Pic^0(Y)$.

\begin{defi}\label{5de}\upshape
We call $T_{\chi_{i,s}}$ a {\em maximal
component} of $V_0(\omega_X)$  if $T_{\chi_{i,s}}$ is maximal for the
inclusion among all $T_{\chi_{j,t}}$.
\end{defi}

 By \cite{CH2}, Theorem 2.3, note that
necessarily, if $\kappa(X)>0$, we have $\dim (T_{\chi_{i,s}})\geq 1$.

We conclude this section with a technical result on the structure of the locus $V_0(\omega_X)$ when moreover  $q(X)=\dim (X)$ and $\kappa(X)>0$.

\begin{prop}\label{5.4}Let $X$ be a
smooth projective variety with maximal Albanese dimension, such that $q(X)=\dim (X)$  and
$\kappa(X)>0$. Let $T_{\chi_{i,s}}$ be a maximal component of
$V_0(\omega_X)$. Then, for any $(j,t)$ such that $\dim
(T_{\chi_{j,t}})\geq 1$, we have $\dim(T_{\chi_{i,s}}\cap
T_{\chi_{j,t}})\geq 1$.
\end{prop}

\begin{proof}Let $\widehat{T}_{\chi_{i,s}}$ and
$\widehat{T}_{\chi_{j,t}}$ be the dual of $T_{\chi_{i,s}}$ and
$T_{\chi_{j,t}}$ respectively. Let $\pi_1$ and $\pi_2$ be the
natural morphisms of abelian varieties $\Alb(X)\rightarrow
\widehat{T}_{\chi_{i,s}}$ and $\Alb(X)\rightarrow
\widehat{T}_{\chi_{j,t}}$. Take an \'{e}tale cover $t:
\widetilde{X}\rightarrow X$ which is induced by an \'{e}tale cover
of $\Alb(X)$ such that $t^*P_{\chi_{i,s}}$ and $t^*P_{\chi_{j,t}}$
are trivial. Let $f_1$ and $f_2$ be the compositions of morphisms
$\pi_1\circ a_X\circ t$ and $\pi_2\circ a_X\circ t$, respectively.
We then take the Stein factorizations of $f_1$ and $f_2$:
\begin{eqnarray*}
\xymatrix{\widetilde{X}\ar[d]^{g_1}\ar[dr]^{f_1}\\
X_1\ar[r]^(.4){h_1}& \widehat{T}_{\chi_{i,s}}}
\xymatrix{&\widetilde{X}\ar[d]^{g_2}\ar[dr]^{f_2}\\
 \qquad\qquad& X_2\ar[r]^(.4){h_2}& \widehat{T}_{\chi_{j,t}}}
\end{eqnarray*}
After modifications, we can assume that $X_1$ and $X_2$ are smooth.
 We claim the following:
\begin{itemize}
\item $h^0(X_1,
\omega_{X_1}\otimes h_1^*P)>0$  for all $P\in T_{\chi_{i,s}}$, and
similarly $h^0(X_2, \omega_{X_2}\otimes h_2^*Q)>0$  for all $Q\in
T_{\chi_{j,t}}$.
\end{itemize}
The argument to prove the claim is due to Chen and Debarre. Let  $c$ be the codimension of $T_{\chi_{i,s}}$ in
$\Pic^0(X)$. By the proof of Theorem 3 of \cite{EL},
$P_{\chi_{i,s}}+T_{\chi_{i,s}}$ is an irreducible component of
$V_c(\omega_X)$. Hence
\begin{eqnarray*}
h^c(\widetilde{X}, \omega_{\widetilde{X}}\otimes t^*P)\geq h^c(X,
\omega_X\otimes P\otimes P_{\chi_{i, s}})>0
\end{eqnarray*}
for any $P\in T_{\chi_{i,s}}$.

Again by the proof of Theorem 3 in \cite{EL}, the dimension of a
general fiber of $g_1$ is also $c$. Since $g_1$ is an algebraic
fiber space, we have $R^cg_{1*}\omega_{\widetilde{X}}=\omega_{X_1}$
(\cite{K3}, Proposition 7.6), and
$$R^cf_{1*}\omega_{\widetilde{X}}=h_{1*}(R^cg_{1*}\omega_{\widetilde{X}})=h_{1*}\omega_{X_1} $$
(\cite{K4}, Theorem 3.4). Moreover, the sheaves
$R^kf_{1*}\omega_{\widetilde{X}}$, satisfy the generic vanishing
theorem (\cite{HAC3}, Corollary 4.2), hence
$V_j(R^kf_{1*}\omega_{\widetilde{X}})\neq T_{\chi_{i,s}}$ for any
$j>0$. Pick $P\in T_{\chi_{i,s}}-\bigcup_{j>0,
k}V_j(R^kf_{1*}\omega_{\widetilde{X}})$, so that
$$H^j(\widehat{T}_{\chi_i,s}, R^kf_{1*}\omega_{\widetilde{X}}\otimes P)=0$$ for all $j>0$ and all $k$.
By the Leray spectral sequence, we have
$$
0\neq h^c(\widetilde{X}, \omega_{\widetilde{X}}\otimes
f_1^*P) = h^0(\widehat{T}_{\chi_{i,s}},
R^cf_{1*}\omega_{\widetilde{X}}\otimes
P) = h^0(\widehat{T}_{\chi_{i,s}},
h_{1*}\omega_{\widetilde{X}}\otimes P).
$$
Hence we conclude the claim by semicontinuity.

If $\dim(T_{\chi_{i,s}}\cap T_{\chi_{j,t}})=0$, the morphism
$$\Alb(X)\xrightarrow{(\pi_1, \pi_2)}\widehat{T}_{\chi_{i,s}}\times \widehat{T}_{\chi_{j,t}}$$
is surjective. Now we consider the following diagram
\begin{eqnarray*}
\xymatrix{
\widetilde{X}\ar[r]^{t}\ar[d]^{g_1}&X\ar[r]^-{a_X}\ar[dr]&\Alb(X)\ar[d]^{\pi_1}\\
X_1\ar[rr]^-{h_1}&&\widehat{T}_{\chi_{i,s}}.}
\end{eqnarray*}
 From the proof of Theorem 3 in \cite{EL}, we know that the fibers of $g_1$ fill up the fibers of $\pi_1$.
Hence we have a surjective morphism
$\widetilde{X}\xrightarrow{(g_1,g_2)}X_1\times X_2$. Since $a_X\circ
t: \widetilde{X}\rightarrow \Alb(X)$ and $(h_{1}, h_{2}):
X_{1}\times X_{2}\rightarrow \widehat{T}_{\chi_{i,s}}\times
\widehat{T}_{\chi_{j,t}}$ are generically finite and surjective, by
the proof of Lemma 3.1 in \cite{CH1}, $K_{\widetilde{X}/X_{1}\times
X_{2}}$ is effective. Therefore, it follows from the claim that
\begin{eqnarray}\label{d16}h^0(\widetilde{X}, \omega_{\widetilde{X}}\otimes
t^*P\otimes t^*Q)>0\end{eqnarray} for all $P\in T_{\chi_{i,s}}$ and
$Q\in T_{\chi_{j,t}}$. Since $t: \widetilde{X}\rightarrow X$ is
birationally equivalent to an \'{e}tale cover of $X$ induced by an
\'{e}tale cover of $\Alb(X)$,
$t_*\cO_{\widetilde{X}}=\bigoplus_\alpha P_{\alpha}$, where
$P_\alpha$ is a torsion line bundle on $X$ for every $\alpha$. Let
$$T=T_{\chi_{i,s}}+T_{\chi_{j,t}}$$ be the abelian variety generated by $T_{\chi_{i,s}}$ and
$T_{\chi_{j,t}}$. Then (\ref{d15}) implies that there exists an $\alpha$
such that
$$P_{\alpha}+T \subset V_0(\omega_X).$$
Since $\dim (T_{\chi_{j,t}})\geq 1$ and $\dim(T_{\chi_{i,s}}\cap
T_{\chi_{j,t}})=0$, we obtain $T_{\chi_{i, s}}\subsetneq T$, contradicting the assumption that $T_{\chi_{i,s}}$ is a
maximal component of $V_0(\omega_X)$. This finishes the proof of the
proposition.
\end{proof}

\section{Varieties with  $q(X)=\dim (X)$ and $0<P_m(X)\leq 2m-2$}

In this section, we prove that the Iitaka model of a smooth projective variety $X$ with  $q(X)=\dim (X)$ and $0<P_m(X)\leq 2m-2$ for some $m\geq 2$ is birational to an abelian variety. We begin with a useful easy lemma (\cite{HAC1}).

\begin{lemm}\label{J4}Let $X$ be a smooth projective variety, let $L$ and $M$
be line bundles on $X$, and let $T\subset \Pic^0(X)$ be an
irreducible subvariety of dimension $t$. If for some positive
integers $a$ and $b$ and all $P\in T$, we have $h^0(X, L\otimes
P)\geq a$ and $h^0(X, M\otimes P^{-1})\geq b$, then $h^0(X, L\otimes
M)\geq a+b+t-1.$
\end{lemm}

Our next result  is a consequence of the proof of Proposition 3.6 and Proposition
3.7 in \cite{CH4}, although not explicitly stated there.

\begin{prop}\label{5.1}Let $X$ be a smooth projective variety with  $q(X)=\dim (X)$ and $0<P_m(X)\leq 2m-2$ for some $m\geq 2$.
Let $f: X\rightarrow Y$ be an algebraic fiber space onto a smooth
projective variety $Y$, which is birationally equivalent to the Iitaka
fibration of $X$. Then $Y$ is birational to an abelian variety.
\end{prop}
\begin{proof} Since we have $0<P_m(X)\leq 2m-2$ for some $m\geq 2$, $a_X$ is
surjective by Theorem \ref{J1}. Since $q(X)=\dim (X)$, we saw in \S\ref{s14}
that $a_X$ and $a_Y$ are both surjective and generically finite. We
then use diagram (\ref{d14}) and the notation of \S\ref{s14}.

If $\kappa(X)=1$, then $Y$ is an elliptic curve, because $a_Y$ is
surjective.

If $\kappa(X)\geq 2$, we use the same argument as in the proof of \cite{CH4},
Proposition 3.6. We claim that
\begin{itemize}
\item[(\dag)] $V_0(\omega_X)\cap f^*\Pic^0(Y)=\{\cO_X\}.$
\end{itemize} Let $\delta$ be the maximal dimension of a component of
$V_0(\omega_X)\cap f^*\Pic^0(Y)$.

If $\delta=0$, $V_0(\omega_X)\cap f^*\Pic^0(Y)=\{\cO_X\}$ by \cite{CH2},
Proposition 1.3.3.

If $\delta\geq 2$, by Lemma \ref{J4}, there exists $P_0\in
f^*\Pic^0(Y)$ such that $$h^0(X, \omega_X^{2}\otimes P_0)\geq
1+1+2-1=3.$$ By Proposition \ref{J2}, $h^0(X, \omega_X^{2}\otimes
P)=h^0(X, \omega_X^{2}\otimes P_0)\geq 3$  for any $P\in
f^*\Pic^0(Y)$. We iterate this process and get $P_m(X)\geq 2m-1$,
which is a contradiction.

If $\delta=1$, there is a $1$-dimensional component $T$ of $
V_0(\omega_X)\cap f^*\Pic^0(Y)$. Let $E=\Pic^0(T)$ and let $g:
X\rightarrow E$ be the induced surjective morphism. By Corollary 2.11 and Lemma
2.13 in \cite{CH4}, for some torsion element $P\in T$, there exist a
line bundle $L$ of degree $1$ on $E$ and an inclusion
$g^*L\hookrightarrow \omega_X\otimes P$, and $P|_F=\cO_F$, where $F$ is
a general fiber of $g$. Since $\kappa(X)\geq 2$, we have
$\kappa(F)\geq 1$. Again by Theorem 3.2 in \cite{CH1},
$\rank(g_*(\omega_X^{ 2}\otimes P^{2}))=P_2(F)\geq 2$. Consider the
exact sequence of sheaves on $E$:
$$0\rightarrow L^{2}\rightarrow g_*(\omega_X^{
2}\otimes P^{2})\rightarrow \cQ\rightarrow 0,$$ where
$\rank(\cQ)\geq 1$. Since $g: X\rightarrow C$ is dominated by $f: X\rightarrow Y$, the Iitaka model of $X$ (i.e. Y) dominates $E$, hence $ g_*(\omega_X^{2}\otimes P^{2})$ is ample
by Corollary \ref{J3}, so is $\cQ$ and $h^0(X, \cQ)\geq 1$. Hence $h^0(X,
\omega_X^{2}\otimes P^{2})\geq 3$.

For any $k\geq 3$, we apply Lemma \ref{5.10} (to be proved below) to get
\begin{eqnarray*}h^0(X, \omega_X^{k}\otimes P^{
k})\geq h^0(X, \omega_X^{k-1}\otimes P^{k-1})+2.\end{eqnarray*}

By induction, we have $h^0(X, \omega_X^{ m}\otimes P^{m})\geq 2m-1$ for all $m\geq 2$. Since $P\in T\subset f^*\Pic^0(Y)$, we have, by
Proposition \ref{J2}, $P_m(X)=h^0(X, \omega_X^{ m}\otimes P^{m})\geq
2m-1$, which is a contradiction.
 We have proved   claim $(\dag)$.

\medskip
Since $X$ and $Y$ are of maximal Albanese dimension, $K_{X/Y}$ is
effective (see the proof of Lemma 3.1 in \cite{CH1}). This implies
$$f^*V_0(\omega_Y)\subset V_0(\omega_X)\cap f^*\Pic^0(Y)=\{\cO_X\},$$
and hence $\kappa(Y)=0$ by Theorem 1 in \cite{CH2}. By Kawamata's
Theorem (\cite{KA}), $a_Y$ is birational.
\end{proof}

We still need to prove the following result used in the proof of the proposition. It
  is an analogue of Corollary 3.2 in \cite{CH4}.

\begin{lemm}\label{5.10}Let $X$ be a smooth projective variety of maximal Albanese dimension with $\kappa(X)\geq 2$.
Suppose that there exist a surjective morphism $g: X\rightarrow C$ onto
an elliptic curve $C$  and  an ample line
bundle $L$ on $C$ with an inclusion $g^*L\hookrightarrow
\omega_X\otimes P_{2}$ for some torsion line bundle $P_2\in
\Pic^0(X)$. Then we have
$$h^0(X, \omega_X^m\otimes P_1\otimes P_2)\geq h^0(X,
\omega_X^{m-1}\otimes P_1)+2,$$ for all torsion line bundles $P_1\in
V_0(\omega_X)$ and all $m\geq 3$.
\end{lemm}
\begin{proof}From the inclusion, we obtain $H^0(X, \omega_X\otimes P_2)\neq 0$, and by items (1) and (2) in Proposition \ref{ch1}, we conclude that $P_2\in V_0(\omega_X, \chi_i)$ for some $i$ and we get an exact sequence of sheaves on $C$:
\begin{eqnarray}\label{d15}
0\rightarrow g_*(\omega_X^{m-1}\otimes P_{1})\otimes
L\hookrightarrow g_*(\omega_X^m\otimes P_1\otimes P_2)\rightarrow
\cQ\rightarrow 0.
\end{eqnarray}
By item (2) in Proposition \ref{ch1}, we have $h^0(X, \omega_X\otimes
P_2^{\vee}\otimes P)\neq 0$ for some $P\in f^*\Pic^0(Y)$ such that
$P_2^{\vee}\otimes P\in V_0(\omega_X, -\chi_i)$. Hence we have an inclusion
$g^*L\hookrightarrow \omega_X^2\otimes P$. Moreover, since $P_2$ is a torsion line
bundle and $V_0(\omega_X)$ is a   subtorus of $\Pic^0(X)$ translated by a torsion point, we may assume that
$P\in f^*\Pic^0(Y)$ is also a torsion line bundle. Therefore the Iitaka model of $X$
dominates $C$. Thus, by Corollary \ref{J3}, both
$g_*(\omega_X^{m-1}\otimes P_{1})$ and $g_*(\omega_X^m\otimes
P_1\otimes P_2)$ are ample, and so is $\cQ$. By Serre duality, for
any ample vector bundle $V$ on $C$, we have $H^1(C, V)=0$. Hence,
Riemann-Roch gives
$$h^0(X, \omega_X^{m-1}\otimes P_1)=h^0(C, g_*(\omega_X^{m-1}\otimes
P_{1}))=\deg(g_*(\omega_X^{m-1}\otimes P_{1})),$$ and \begin{eqnarray*}h^0(
\omega_X^m\otimes P_1\otimes P_2)&=&h^0(C, g_*(\omega_X^{m-1}\otimes
P_1)\otimes L)+h^0(C, \cQ).\end{eqnarray*} Let $F$ be a connected component of
a general fiber of $g$. Since $\kappa(X)\geq 2$, we have
$\kappa(F)\geq 1$ by the easy addition formula (Corollary 1.7 in
\cite{Mo}). Hence we have $P_2(F)\geq 2$ by Theorem 3.2 in
\cite{CH1} (see also Remark \ref{5.8}). Since $P_1\in
V_0(\omega_X)$, we have $h^0(X, \omega_X^m\otimes P_1\otimes
P_2)\geq h^0(X, \omega_X^{m-1}\otimes P_2)> 0.$ Hence we have
$h^0(F,\omega_F^m\otimes P_1\otimes P_2)\neq 0$. Then, by Lemma
\ref{5.11} and Proposition \ref{ch1}, there exists $P'\in \Pic^0(F)$ which is pulled
back by the Iitaka fibration of
$F$ such that $(P_1\otimes P_2)|_F\otimes P'\in V_0(\omega_F)$. 
On the other hand, since $P_1\otimes P_2$ is torsion, we have $h^0(F, \omega_F^m\otimes P_1\otimes
P_2)=h^0(F, \omega_F^m\otimes P_1\otimes P_2\otimes P^{'})$ by Proposition
\ref{J2}. Therefore, we conclude
$$
h^0(F,
\omega_F^m\otimes P_1\otimes P_2) =  h^0(F, \omega_F^m\otimes
P_1\otimes P_2\otimes P^{'})  \geq   P_{m-1}(F) \geq
2,
$$
where the last inequality holds since $m\geq 3$.
Hence,
$$\rank(g_*(\omega_X^m\otimes P_1\otimes P_2))=h^0(F,
\omega_F^m\otimes P_1\otimes P_2)\geq 2.$$

Since $P_1\in V_0(\omega_X)$ by assumption, we have
$\rank(g_*(\omega_X^{m-1}\otimes P_{1}))\geq 1$.

If $\rank(g_*(\omega_X^{m-1}\otimes P_{1}))\geq 2$, we have
\begin{eqnarray*}h^0(C, g_*(\omega_X^m\otimes P_1\otimes P_2))&\geq&
h^0(C, g_*(\omega_X^{m-1}\otimes P_{1})\otimes
L)\\&\geq&\deg(g_*(\omega_X^{m-1}\otimes
P_{1}))+\rank(g_*(\omega_X^{m-1}\otimes P_{1}))\\&\geq &h^0(X,
\omega_X^{m-1}\otimes P_{1})+2. \end{eqnarray*}

If $\rank(g_*(\omega_X^{m-1}\otimes P_{1}))=1$, $\cQ$ has
rank $\geq 1$. Since $\cQ$ is ample, $h^0(C, \cQ)\geq 1$. We also
have
\begin{eqnarray*}h^0(X, \omega_X^m\otimes P_1\otimes P_2)&=& h^0(C,
g_*(\omega_X^{m-1}\otimes P_{1})\otimes L)+h^0(C, \cQ)\\&\geq &
h^0(X, \omega_X^{m-1}\otimes P_1)+ \rank(\omega_X^{m-1}\otimes
P_1)+1 \\&=& h^0(X, \omega_X^{m-1}\otimes P_1)+2.\end{eqnarray*}
Hence the lemma is proved.
\end{proof}

Under the hypotheses of Proposition \ref{5.1}, it turns out that when $m\ge 4$, we can bound the Kodaira dimension of $X$ by 1 (the case $m=2$ is the object of the next section; when $m=3$, the bound $\kappa(X)\le 2$ was obtained in \cite{CH4} and there are examples when there is equality).

\begin{theo}\label{4.11}Let $X$ be a smooth projective variety with $q(X)=\dim (X)$ and $0< P_m(X)\leq 2m-2$  for some $m\geq 4$. Then $\kappa(X)\leq 1$.
\end{theo}

\begin{proof}By Theorem \ref{J1}, the Albanese morphism $a_X: X\rightarrow \Alb(X)$
is surjective and hence generically finite. We then use diagram (\ref{d14}). By Proposition \ref{5.1}, we may assume that $Y$, the image of
the Iitaka fibration of $X$, is an abelian variety.

We   assume $\kappa(X)\geq 2$
and under this assumption we will deduce a contradiction.

Let $T_{\chi_{1,s}}$ be a maximal component of $V_0(\omega_X)$ in
the sense of Definition \ref{5de}. If $\dim (T_{\chi_{1,s}})=1$, by
Proposition \ref{5.4}, we conclude that $T_{\chi_{i,t}}\subset
T_{\chi_{1,s}}$ for any $(i,t)$ such that $\dim (T_{\chi_{i,t}})>0$.
By Theorem 2.3 in \cite{CH2}, $\Pic^0(Y)=T_{\chi_{1,s}}$. Then $\dim
(Y)=\dim ( \Pic^0(Y))=1$, which contradicts our assumption   that
$\kappa(X)\geq 2$. Hence we get   $\dim (T_{\chi_{1,s}})\geq 2$.

We then iterate Lemma \ref{J4} to get $$h^0(X, \omega_X^{m-i}\otimes P_{\chi_{1,s}}^{m-i})
\geq (m-i-1)\dim (T_{\chi_{1,s}})+1.$$
By Proposition \ref{J2}, we have
\begin{eqnarray}\label{d18}h^0(X, \omega_X^{m-i}\otimes P_{\chi_{1,s}}^{m-i}\otimes f^*P)
\geq (m-i-1)\dim (T_{\chi_{1,s}})+1,
\end{eqnarray}
for all $0\leq i\leq m-2$ and all $P\in\Pic^0(Y)$. According to item (2) in
Proposition \ref{ch1}, $V_0(\omega_X, -(m-1)\chi_1)$ is not empty, namely
there exists $P_0\in \Pic^0(Y)$ such that $h^0(X, \omega_X\otimes
P_{\chi_{1,s}}^{-(m-1)}\otimes P_0)>0$. Thus $h^0(X, \omega_X^m\otimes P_0)\geq h^0(X, \omega_X^{m-1}\otimes P_{\chi_{1,s}}^{m-1})$. Again by Proposition \ref{J2}, we have
$$P_m(X)=h^0(X, \omega_X^m\otimes P_0)\geq h^0(X, \omega_X^{m-1}\otimes P_{\chi_{1,s}}^{m-1}).$$
We also have
$$2m-2\geq P_m(X) \ge h^0(X,
\omega_X^{m-1}\otimes P_{\chi_{1,s}}^{m-1}) \geq   (m-2)\dim
(T_{\chi_{1,s}})+1,$$ where the last inequality holds by
taking $i=1$ in (\ref{d18}). Hence we deduce that $\dim (T_{\chi_{1,s}})=2$.

\medskip\noindent{\bf Claim 1:}  $(m-1)\chi_1=0$.

If $(m-1)\chi_1\neq 0$, by item (3) in Proposition \ref{ch1}, there exists a torsion
point $P_{-(m-1)\chi_{1, t}}\in \Pic^0(X)$ such that
$P_{-(m-1)\chi_{1, t}}+T_{-(m-1)\chi_{1, t}}\subset V_0(\omega_X)$
with $\dim (T_{-(m-1)\chi_{1, t}})\geq 1$.

If $\dim (T_{-(m-1)\chi_{1, t}})\geq 2$, by (\ref{d18}) (let $i=1$) and
Lemma \ref{J4}, we get $P_m(X)\geq 2m-3+1+2-1=2m-1$, which is a
contradiction.

Hence $\dim (T_{-(m-1)\chi_{1, t}})=1$. Let
$C=\widehat{T}_{-(m-1)\chi_{1, t}}$ and let $\pi: \Alb(X)\rightarrow
C$ be the dual of the inclusion
$T_{-(m-1)\chi_{1, t}}\hookrightarrow\Pic^0(X)$. Then we set
$f=\pi\circ a_X$ as in the following commutative diagram:
\begin{eqnarray*}
\xymatrix{
X\ar[d]^{a_X}\ar[dr]^{f}\\
\Alb(X)\ar[r]^{\pi}&C.}
\end{eqnarray*}
Since we assume   $\kappa(X)\geq 2$ and $\dim (T_{-(m-1)\chi_{1, t}})=1$, we have $V_0(\omega_X)\neq \Pic^0(X)$, therefore $\chi(\omega_X)=0$. By Lemma 2.10 and Corollary 2.11 in \cite{CH4}, there exists
an ample line bundle $L$ on $C$ such that $f^*L\hookrightarrow
\omega_X\otimes P_{-(m-1)\chi_{1, t}}$. We then apply Lemma
\ref{5.10} to conclude that
\begin{eqnarray*}P_m(X)&=&h^0(X, \omega_X^m\otimes P_{\chi_{1,s}}^{m-1}\otimes P_{-(m-1)\chi_{1, t}})\\
&\geq& h^0(X, \omega_X^{(m-1)}\otimes
P_{\chi_{1,s}}^{m-1})+2\\&\geq& 2m-1,\end{eqnarray*} where the last inequality holds by (\ref{d18}). This is a
contradiction. We have proved   Claim 1.
\medskip

Let $\overline{G}$ be defined as in the beginning of \S\ref{s14}.

\medskip\noindent{\bf Claim 2:}  $\overline{G}\simeq \mathbb{Z}/2$,
namely $\overline{G}$ contains only one nonzero element $\chi_1$. In
particular, by Claim 1, $m$ is an odd number.

Assuming the claim is not true, there exists $0\neq \chi_2\in
\overline{G}$ such that $(m-2)\chi_1+\chi_2\neq 0$. According to item (3) in
Proposition \ref{ch1}, there exists $P_{\chi_{2,t}}+T_{\chi_{2,t}}\subset
V_0(\omega_X, \chi_2)$ with $\dim (T_{\chi_{2,t}})\geq 1$. Then as in the proof of Claim 1, by Lemma \ref{J4} and Lemma \ref{5.10},
we conclude
$$
h^0(X, \omega_X^{m-1}\otimes P_{\chi_{1,s}}^{m-2}\otimes P_{\chi_{2,t}}) \ge  h^0(X, \omega_X^{m-2}\otimes P_{\chi_{1,s}}^{m-2})+2 \ge 2m-3,$$
 where the last inequality holds because
of (\ref{d18}).

Since $(m-2)\chi_1+\chi_2\neq 0$, we may repeat the above process to
get
$$
P_m(X) \geq  h^0(X, \omega_X^{m-1}\otimes P_{\chi_{1,s}}^{m-2}\otimes P_{\chi_{2,t}})+2\\
 \geq  2m-1,$$
which is a contradiction. Hence we have proved Claim 2.\\

As $m\geq 4$ is odd, $m-2\geq 3$ and $(m-3)\chi_{1}=0$. By (\ref{d18}) (with $i=3$), $P_{m-3}(X)\geq 2m-7$. Since $\kappa(X)\geq 2$, by
Proposition \ref{J2} and Lemma \ref{J4}, we have
\begin{eqnarray*}
2m-2 \geq P_{m}(X) &\geq &P_{m-3}(X)+P_3(X)+\kappa(X)-1\\&\geq
&2m-6+P_3(X).
\end{eqnarray*}
Hence $P_3(X)\leq 4$. According to Chen and Hacon's classification
of these varieties (\cite{CH4}, see Theorems 1.1 and   1.2)
and Claim 2, the only possibility is that $X$ is a double cover of
its Albanese variety and $\kappa(X)=2$, as described in Example 2 in
\cite{CH4}. Namely, there exists an algebraic fiber space
\begin{eqnarray*}q: \Alb(X)\rightarrow S
\end{eqnarray*} from an abelian variety of dimension $\geq 3$ to an abelian surface, and $a_X: X\rightarrow \Alb(X)$
is birational to a double cover of $\Alb(X)$ such that
$a_{X*}\cO_X=\cO_{\Alb(X)}\oplus (q^*L\otimes P)^{\vee}$, where $L$
is an ample divisor of $S$ with $h^0(S, L)=1$ and $P\in \Pic^0(A)\moins\Pic^0(S)$ and $2P\in\Pic^0(S)$. However, for such a
variety, we have the inclusion of sheaves $a_X^*(q^*L\otimes
P)\hookrightarrow \omega_X$ (see the proof of Claim 4.6 in
\cite{CH4}). Thus, as $m\geq 4$ is odd,
\begin{eqnarray*}P_m(X)&=&h^0(X, \omega_X^m)\\
&\geq& h^0(\Alb(X), q^*L^{m}\otimes P^m\otimes a_{X*}\cO_X)\\
&=&h^0(\Alb(X), q^*L^{m-1}\otimes P^{m-1})\\&=&
(m-1)^2 > 2m-2,\end{eqnarray*} which is a contradiction. This
concludes the proof of Theorem \ref{4.11}.
\end{proof}

\section{Varieties with  $q(X)=\dim (X)$ and  $P_2(X)=2$}

In this section, we describe  explicitly all  smooth projective varieties $X$ with $q(X)=\dim (X)$ and  $P_2(X)=2$.
 We first show that the Iitaka model of   $X$   is an elliptic curve. In particular, $\kappa(X)=1$.

\begin{prop}\label{5.5}Let $X$ be a smooth projective variety with $q(X)=\dim (X)$ and  $P_2(X)=2$. Assume that $f: X\rightarrow Y$ is a birational
model of the Iitaka fibration of $X$ and $Y$ is a smooth projective
variety. Then $Y$ is an elliptic curve.
\end{prop}
\begin{proof}
We use diagram (\ref{d14}). By Theorem \ref{J1}, $a_X$ is surjective and
hence generically finite. By Proposition \ref{5.1}, we may assume
that $a_Y: Y\rightarrow \Alb(Y)$ is an isomorphism.

If $\dim (Y)=1$, $Y$ is an elliptic curve and we are done. We now assume
that $\dim (Y)\geq 2$ and deduce a contradiction.

Let $T_{\chi_{i,s}}$ be a maximal component of $V_0(\omega_X)$. By
the claim $(\dag)$ in the proof of  Proposition \ref{5.1}, we know
that $P_{\chi_i}\notin f^*\Pic^0(Y)$. By item (3) of Proposition \ref{ch1},
there exist a torsion line bundle $P_{-\chi_{i,t}}\in
P_{\chi_i}^{\vee}+f^*\Pic^0(Y)$ and a positive-dimensional abelian
subvariety $T_{-\chi_{i,t}}\subset f^*\Pic^0(Y)$ such that
$P_{-\chi_{i,t}}+T_{-\chi_{i,t}}$ is a connected component of
$V_0(\omega_X)$. Let $T$ be the neutral component of $T_{\chi_{i,s}}\cap T_{-\chi_{i,t}}$. By
Proposition \ref{5.4}, $\dim (T)\geq 1$.

If $\dim (T)\geq 2$, then by Lemma \ref{J4}, $$h^0(X,
\omega_X^{2}\otimes P\otimes Q)\geq 1+1+2-1=3,$$ for all $P\in
P_{\chi_{i,s}}+T$ and all $Q\in P_{-\chi_{i,t}}+T$. Since
$P_{\chi_{i,s}}\otimes P_{-\chi_{i,t}}\in f^*\Pic^0(Y)$, we
obtain, by Proposition \ref{J2}, $P_2(X)\geq 3$, which is a
contradiction.

Hence $T$ is an elliptic curve and we denote by $\widehat{T}$ its
dual.
 There exists a projection $\pi: Y\rightarrow \widehat{T}$. We then
consider the commutative diagram:
\begin{eqnarray*}
\xymatrix{
X\ar[d]^f\ar[dr]^{\bar{f}}\\
Y\ar[r]^{\pi}&\widehat{T} }
\end{eqnarray*}
and define
\begin{eqnarray*}F_1&=&\bar{f}_*(\omega_X\otimes
P_{\chi_{i,s}}),\\ F_2&=&\bar{f}_*(\omega_X\otimes P_{-\chi_{i,t}}),
\\F_3&=&\bar{f}_*(\omega_X^2\otimes P_{\chi_{i,s}}\otimes
P_{-\chi_{i,t}}).
\end{eqnarray*}
These are vector bundles on the elliptic curve $\widehat{T}$ and by
Corollary 3.6 in \cite{V} and Corollary \ref{J3}, $F_1$ and $F_2$
are nef and $F_3$ is ample.xxx

Since $P_{\chi_{i,s}}\otimes P_{-\chi_{i,t}}\in f^*\Pic^0(Y)$ and
$f: X\rightarrow Y$ is a model of the Iitaka fibration of $X$, we
have $$2=P_2(X)=h^0(X, \omega_X^{2}\otimes P_{\chi_{i,s}}\otimes
P_{-\chi_{i,t}})=h^0(\widehat{T}, F_3).$$

There exists a natural morphism
$$F_1\otimes F_2\xrightarrow{\upsilon} F_3,$$ corresponding to the
multiplication of sections. Let $R_1$, $R_2$, and $R_3$ be the
respective ranks of $F_1$, $F_2$, and $F_3$.
 I claim:
 \begin{itemize}\item[$\spadesuit$] $R_3>\min\{R_1,
R_2\}$.
\end{itemize}

Indeed if $R_1\geq 2$ and $R_2\geq 2$, then by Lemma \ref{J4},
$R_3\geq R_1+R_2-1$. If either $R_1$ or $R_2$ is $1$, we just need
to prove   $R_3\geq 2$. Let $f|_{X_t}: X_t\rightarrow Y_t$ be the
restriction of $f$ to a general fiber of $\bar{f}$. Since $f:
X\rightarrow Y$ is a model of the Iitaka fibration of $X$, fixing an
ample divisor $H$ on $Y$, there exists an integer $N>0$ such
that $NK_X-H$ is effective. Hence $(NK_X-H)|_{X_t}\succeq 0$,
therefore  the Iitaka model of $X_t$ dominates $Y_t$. Indeed,
$f|_{X_t}: X_t\rightarrow Y_t$ is a birational model of the Iitaka
fibration of $X_t$ since a general fiber of $f|_{X_t}$ is isomorphic
to a general fiber of $f$ which is birational to an abelian variety.
As we have assumed   $\dim (Y)\geq 2$, we have $\dim (Y_t)\geq 1$.
Thus $X_t$ is of maximal Albanese dimension and $\kappa(X_t)\geq 1$,
hence $P_2(X_t)\geq 2$ (\cite{CH1}, Theorem 3.2). Since
$(P_{\chi_{i,s}}\otimes P_{-\chi_{i,t}})|_{X_t}$ is pulled back from
$Y_t$, we have
\begin{eqnarray*}R_3=h^0(X_t, (\omega_X^2\otimes P_{\chi_{i,s}}\otimes
P_{-\chi_{i,t}})|_{X_t})=P_2(X_t)\geq 2,\end{eqnarray*} where the
second equality holds again because of Proposition \ref{J2}. This proves the claim $\spadesuit$.
\medskip

Consider the Harder-Narasimhan filtration  of $F_1$ (resp. $F_2$) and
let $G_1$ (resp. $G_2$) be the unique maximal subbundle of $F_1$
(resp. $F_2$) of largest slope. By definition, $G_1$ and $G_2$ are
semistable. Let $r_1$ and $r_2 $ be their respective ranks.
 I claim:
\begin{itemize}
\item[$\clubsuit$]  $r_1>0$ and $r_2>0$ and therefore $G_1$ and $G_2$ are ample.
\end{itemize}
If $\deg(G_1)=0$, we conclude from the definition of $G_1$ that
$0=\mu(G_1)\geq \mu(F_1)$. Since $F_1$ is nef,
$\deg(F)=\mu(F)=0$, hence $V_0(F_1)=V_1(F_1)$. By the generic vanishing theorem (see for example \cite{HAC3}), $V_1(F_1)$ is finite. However, since $T\subset
T_{\chi_{i,s}}$, we have $h^0(F_1\otimes P)>0$ for all $P\in T$,
which is a contradiction. So $r_1>0$  and similarly,
$r_2>0$. Since $G_1$ and $G_2$ are semistable, they are ample
(see the Main Claim in the proof of Theorem 6.4.15 in \cite{Laz}). This proves the claim $\clubsuit $.\medskip

 Set
$G_3=\upsilon(G_1\otimes G_2)$ and let $r_3$ be its rank. Again by
Lemma \ref{J4}, we have $$r_3\geq r_1+r_2-1\geq \max\{r_1, r_2\}.$$
Since $G_1$ and $G_2$ are semistable and ample, so is $G_1\otimes
G_2$ (\cite{Laz}, Corollary 6.4.14). Therefore the slopes
satisfy
$$\mu(G_3)\geq \mu(G_1\otimes G_2)=\mu(G_1)+\mu(G_2),$$ and $G_3$ is also ample.

We then apply Riemann-Roch, \begin{eqnarray*}h^0(\widehat{T},
G_3)&\geq& r_3(\mu(G_1)+\mu(G_2))\\
&\geq& \deg(G_1)+\deg(G_2)\\
&\geq& 2,
\end{eqnarray*}
where the second inequality holds because $r_3\geq \max\{r_1, r_2\}$
and the third inequality holds because $\deg(G_1)>0$ and
$\deg(G_2)>0$.

Since $G_3$ is a subbundle of $F_3$ and $h^0(\widehat{T}, F_3)=2$,
we have $h^0(\widehat{T}, G_3)=2$, hence all the inequalities above
should be equalities. In particular, $r_3=r_1=r_2$. Hence by the
claim $\spadesuit$, $r_3\leq \min\{R_1, R_2\}<R_3$. Therefore,
$F_3/G_3$ is a sheaf of rank $\geq 1$. Since $F_3$ is ample, so is
$F_3/G_3$, hence $h^0(\widehat{T}, F_3/G_3)\geq 1$. Since $G_3$ is
ample, $h^1(\widehat{T}, G_3)=0$. Hence,
$$h^0(\widehat{T}, F_3)=h^0(\widehat{T}, G_3)+h^0(\widehat{T},
F_3/G_3)\geq 3,$$ which is a contradiction. Thus $\dim (Y)=1$. This
finishes the proof of Proposition \ref{5.5}.
\end{proof}

\begin{rema}\label{5.8}\upshape
It is easy to see that combining Proposition \ref{5.4} and the proof
of Proposition \ref{5.5}, one can give another proof of Chen and
Hacon's characterization of abelian varieties (\cite{CH1}): {\em
  a smooth projective variety $X$ with maximal Albanese
dimension and $ P_2(X)=1$  is birational to an abelian
variety.}
\end{rema}

The following theorem is the main result of this article. It describes explicitly all  smooth projective varieties $X$ with $q(X)=\dim (X)$ and  $P_2(X)=2$.

\begin{theo}\label{5}Let $X$ be a smooth projective variety with $P_2(X)=2$
and $q(X)=\dim (X)$. Then $\kappa(X)=1$ and $X$ is birational to a
quotient $(K\times C)/G$, where $K$ is an abelian variety and $C$ is
a smooth projective curve, $G$ is a finite group that acts diagonally and freely on
$K\times C$, and $C\rightarrow C/G$ is branched at $2$ points.
\end{theo}

\begin{proof}Since we know by Proposition \ref{5.5} that $Y$ is an elliptic curve, the proof is parallel to the proof of Theorem 5.1 in
\cite{CH4}. By \cite{KA}, Theorem 13, there exists a curve $C$ of genus
$g\geq 2$, an abelian variety $\KK$, and a finite abelian group $G$,
which acts faithfully on $C$ and $\KK$, such that $X$ is birational
to $(\KK\times C)/G$, where $G$ acts diagonally and freely on
$\KK\times C$.

We then consider the induced morphism $\phi: C\rightarrow C/G=Y$.
Following \cite{Be}, \S VI.12, we have $$2=P_2(X)=\dim (
H^0(C, \omega_C^{2})^G)=h^0(Y, \cO_Y(\sum_{P\in
Y}\big\lfloor2(1-\frac{1}{e_P})\big\rfloor P)),$$ where $P$ is a
branch point of $\phi$, and $e_P$ is the ramification index of a
ramification point lying over $P$.

Since $\big\lfloor2(1-\frac{1}{e_P}) \big\rfloor=1$ for any branch
point $P$, we have only two branch points.
\end{proof}
\begin{exam}\upshape
Let $C$ be a bi-elliptic curve of genus $2$, let $\phi: C\rightarrow E$ be
the morphism such that $\phi$ is branched at two points, and let $\tau$ be
the induced involution. Take  an abelian variety  $K$ and set
$G=\mathbb{Z}_2$. Let $G$ act on $C$ by $\tau$ and on $K$ by
translation by a point of order $2$. Set $X=(K \times  C)/G$, where
$G$ acts diagonally. Then $P_2(X)=h^0(C, \omega_C^{2})^{\tau}=2$ (this construction actually gives all varieties with $q(X)=\dim (X)$ and
$P_3(X)=2$; see \cite{HAC1}).
\end{exam}
\begin{rema}\upshape
The family of varieties with   $q(X)=\dim (X)$ and  $P_2(X)=2$ is not bounded   (see Example 1 in \cite{CH4}).
\end{rema}

 \subsection*{Acknowledgements}
I am grateful to my thesis advisor, O. Debarre, for
his generous help.

\end{document}